\documentclass[a4paper]{amsart}
\usepackage{graphicx}
\usepackage{amssymb} 
\usepackage{stmaryrd}
\usepackage[mathcal]{euscript}
\usepackage{tikz}
\usepackage{tikz-cd}
\usepackage{hyperref}
\hypersetup{%
  bookmarksnumbered=true,%
  colorlinks=true,%
  linkcolor=blue,%
  citecolor=blue,%
  filecolor=blue,%
  menucolor=blue,%
  urlcolor=blue,%
  bookmarksopen=true,%
  bookmarksdepth=2,%
  pageanchor=true}

\title{On the symmetry of the finitistic dimension}

\author{Henning Krause}
\address{Fakult\"at f\"ur Mathematik\\
Universit\"at Bielefeld\\ D-33501 Bielefeld\\ Germany}
\email{hkrause@math.uni-bielefeld.de}

\theoremstyle{plain}
\newtheorem{thm}{Theorem}
\newtheorem{prop}[thm]{Proposition}
\newtheorem{lem}[thm]{Lemma} 

\newtheorem{cor}[thm]{Corollary}

\theoremstyle{definition}

\theoremstyle{remark}
\newtheorem{rem}[thm]{Remark}

\numberwithin{equation}{thm}

\newcommand{\Coker}{\operatorname{Coker}}
\newcommand{\colim}{\operatorname*{colim}}

\newcommand{\End}{\operatorname{End}}

\newcommand{\findim}{\operatorname{fin{.}dim}}
\newcommand{\Findim}{\operatorname{Fin{.}dim}}

\newcommand{\Hom}{\operatorname{Hom}}

\newcommand{\Mod}{\operatorname{Mod}}

\newcommand{\pdim}{\operatorname{proj{.}dim}}

\newcommand{\proj}{\operatorname{proj}}

\newcommand{\rad}{\operatorname{rad}}

\newcommand{\soc}{\operatorname{soc}}

\renewcommand{\top}{\operatorname{top}}

\newcommand{\op}{\mathrm{op}}

\newcommand{\iso}{\xrightarrow{\raisebox{-.4ex}[0ex][0ex]{$\scriptstyle{\sim}$}}}

\newcommand{\longiso}{\xrightarrow{\ \raisebox{-.4ex}[0ex][0ex]{$\scriptstyle{\sim}$}\ }}

\newcommand{\lto}{\longrightarrow}
\newcommand{\smatrix}[1]{\left[\begin{smallmatrix}#1\end{smallmatrix}\right]}

\newcommand*{\intref}[2]{\def\tmp{#1}\ifx\tmp\empty\hyperref[#2]{\ref*{#2}}\else\hyperref[#2]{#1~\ref*{#2}}\fi}

\def\M{\mathcal M}
\def\P{\mathcal P}

\def\bbN{\mathbb N}

\def\a{\alpha}
\def\b{\beta}
\def\e{\varepsilon}

\def\p{\phi}

\def\La{\Lambda}
\def\Si{\Sigma}

\begin{document}



\begin{abstract}
  For any ring we propose the construction of a cover which
  increases the finitistic dimension on one side and decreases the
  finitistic dimension to zero on the opposite side.  This complements
  recent work of Cummings.
\end{abstract}

\date{11 November, 2022}
\dedicatory{Dedicated to the memory of Helmut Lenzing.}
\maketitle

The finitistic dimension is a homological invariant of a ring which is
conjectured to be finite when the ring is a finite dimensional algebra
over a field \cite{Ba1960}. In recent work \cite{Cu2022} Cummings
introduces for any finite dimensional algebra a related algebra; its
purpose is to increase the finitistic dimension on one side and to
decrease the finitistic dimension to zero on the opposite side. In
this note we propose the construction of such an \emph{asymmetric
  cover} for any ring and we establish the same properties. This specialises
to Cummings' construction for finite dimensional algebras over an
algebraically closed field and yields examples of rings such that the
finitistic dimension is infinite while the finitistic dimension of
the opposite ring is zero. We need to distinguish between the small
and the big finitistic dimension but our results cover both
cases.\smallskip

\begin{center}
  \textasteriskcentered \qquad \textasteriskcentered \qquad \textasteriskcentered
\end{center}\smallskip

Let $A$ be an associative ring. We consider the category of (right)
$A$-modules and identify left $A$-modules with modules over the
opposite ring $A^\op$.  For a module $M$ we write $\rad M$ for its
\emph{radical} and $\soc M$ for its \emph{socle}. We set
$\top M=M/\rad M$. The functor
\[(-)^*:=\Hom_A(-,A)\]  yields a duality between
right and left $A$-modules.
We consider the \emph{trivial extension}
\[T(A):=A\ltimes A^\natural\] which is given by the bimodule
$A^\natural:={_AA_A}$.
This ring is by definition the abelian group $T(A)=A\oplus A^\natural$ with
multiplication given by the formula
\[(x,y)\cdot (x',y')=(xx',xy'+yx').\]
Note that $A^\natural$ is a two-sided ideal with $T(A)/A^\natural\iso A$.

\begin{lem}\label{le:T(A)}
Let $A$ be a semisimple ring. Then \[\rad T(A)= A^\natural=
  \soc T(A)\qquad\text{and}\qquad \top
T(A)\cong \soc T(A).\]
\end{lem}
\begin{proof}
  The first assertion is clear. Left multiplication with $(0,1)$
  gives a map $T(A)\to T(A)$ which induces an isomorphism $ \top
T(A)\iso\soc T(A)$
\end{proof}

For a ring $A$ we denote by $\Si(A)$ the set of isomorphism classes of
simple $A$-modules. Set
\[\bar S:=\coprod_{S\in\Si(A)} S\qquad\text{and}\qquad\bar
  A:=\prod_{S\in\Si(A)}T(\End_A(S)).\] We view $\bar S$ as
$\bar A$-$A$-bimodule, with left action via
$T(\End_A(S))\twoheadrightarrow \End_A(S)$ for each $S$ in $\Si(A)$,
and consider the triangular matrix ring
\[\tilde A:=\begin{bmatrix}A&0\\\bar S&\bar A\end{bmatrix}.\]
The idempotents $e=\smatrix{1&0\\0&0}$ and  $f=\smatrix{0&0\\0&1}$
provide an $\tilde A$-module decomposition $\tilde A=P\oplus Q$, where
\[P:=e\tilde A\cong A\qquad\text{and}\qquad
  Q:=f \tilde A\cong \bar S\oplus \bar A.\]
We call
$\tilde A$ a \emph{cover} of $A$ because the idempotent $e\in\tilde A$
yields an isomorphism
\[\End_{\tilde A}(P)\cong e \tilde A e\cong A.\] The following lemma expresses the
distinct property of the cover $\tilde A$, namely that every simple
$\tilde A$-module embeds into $\tilde A$.

\begin{lem}\label{le:top(A)}
 We have $S^*\neq 0$ for every simple $\tilde A$-module $S$.
\end{lem}
\begin{proof}
  We claim that each simple $\tilde A$-module arises as the image of a
  morphism $\tilde A=P\oplus Q\to Q$, using that
  \[\Hom_{\tilde A}(P,Q)\cong e\tilde A f=\bar S\qquad\text{and}\qquad
    \Hom_{\tilde A}(Q,Q)\cong f\tilde A f=\bar A.\]
 With Lemma~\ref{le:T(A)} we compute $\soc Q$ and obtain a
 decomposition into simples: 
 \[\soc Q=\bar S\oplus \soc \bar A\cong
\coprod_{S\in\Si(A)}   \big(S\oplus \End_A(S)^\natural\big).\] A
 simple $\tilde A$-module $T$ comes either with a nonzero map $P\to T$
 or a nonzero map $Q\to T$. In the first case $T$ identifies with a
 simple $A$-module via the inclusion $A\to \tilde A$ given by
 $x\mapsto \smatrix{x&0\\0&1}$, and therefore with a summand of
 $\bar S\subseteq \soc Q$. In the second case $T$ identifies with a
 simple $\bar A$-module via the inclusion $\bar A\to \tilde A$ given
 by $x\mapsto \smatrix{1&0\\0&x}$, and therefore with a summand of
 $\soc\bar A\subseteq \soc Q$. In any case one obtains a
 monomomorphism $T\to Q\hookrightarrow \tilde A$.
\end{proof}

Let $\P(A)$ denote the class of $A$-modules $M$ that admit a finite
resolution
\[0\lto P_n\lto \cdots \lto P_1 \lto P_0\lto M\lto 0\]
with all $P_i$ finitely generated projective. We denote by 
\[\findim A:=\sup\{\pdim M\mid M\in\P(A)\}\]
the \emph{small finitistic dimension} of $A$; this is a slight
variation of the usual definition which seems natural as the
modules in $\P(A)$ are precisley the ones which become compact (or
perfect) when viewed as an object in the derived category of $A$.
  
The following lemma is \cite[Lemma~7.2.8]{Kr2022} and its proof
is sketched for the convenience of the reader; cf.\ the discussion in
\cite[\S5]{Ba1960}.

\begin{lem}\label{le:findim-zero}
  For a ring $A$ we have $\findim A=0$ if and only if $M^*\neq 0$
  for every finitely presented $A^\op$-module $M$.
\end{lem}
\begin{proof}
  Let $\proj A$ denote the category of finitely generated projective
  $A$-modules.  The condition $\findim A=0$ means that every
  monomorphism in $\proj A$ splits. The duality
  \[(-)^*\colon (\proj A)^\op\longiso\proj(A^\op)\]
  translates this into the condition on finitely presented 
  $A^\op$-modules.
\end{proof}

\begin{thm}\label{th:tildeA}
  For a  ring $A$ we have
  \[\findim \tilde A\geq\findim A\qquad\text{and}\qquad \findim
    \tilde A^\op=0.\]
  \end{thm}
  \begin{proof}
    The idempotent $e\in\tilde A$ with $e\tilde A e\cong A$ gives rise
    to a fully faithful functor
    \[-\otimes_A e\tilde A\colon\Mod A\lto\Mod\tilde A\]
which is exact and maps projectives to projectives; also it
preserves finite generation. This yields the first
assertion. The second assertion follows from
Lemma~\ref{le:findim-zero} because we have
$M^*\neq 0$ for every finitely presented $\tilde A$-module $M$ by  Lemma~\ref{le:top(A)}. 
\end{proof}
\begin{center}
   \textasteriskcentered \qquad \textasteriskcentered \qquad \textasteriskcentered
\end{center}\smallskip

There is a somewhat more natural construction of a cover when the ring
$A$ is semilocal.  Recall that $A$ is \emph{semilocal} if the ring
$A/\rad A$ is semisimple. In this case we have an idempotent
$\e\in A/\rad A$ and a Morita equivalence
\[A/\rad A  \;\sim\; \e(A/\rad A)\e
  \;\cong\;\prod_{S\in\Si(A)}\End_A(S).\]
We set
\[\widetilde A:=\begin{bmatrix}A&0\\ A/\rad A&T(A/\rad A)\end{bmatrix}\]
and this is closely related to the cover $\tilde A$ via the idempotent $\tilde\e=\smatrix{1&0\\0&(\e,0)}$.

\begin{lem}
For a semilocal ring $A$ we have a Morita equivalence
 \[\widetilde A\;\sim\; \tilde\e\widetilde A\tilde\e\;\cong\;\tilde A.\]
\end{lem}
\begin{proof}
  We use some general facts.  Let $\La=\smatrix{A&0\\M&B}$ be a
  triangular matrix ring and $e\in B$ an idempotent such that $B$ and
  $e Be$ are Morita equivalent. Set $\bar e=\smatrix{1&0\\0&e}$. Then
  $\La$ and $\bar e\La\bar e=\smatrix{A&0\\ e M&eBe}$ are Morita
  equivalent. Also, the trivial extensions $T(B)$ and
  $T(eBe)=(e,0) T(B)(e,0)$ are Morita equivalent.
  \end{proof}

  We may identify $\widetilde A$ with $\tilde A$, as we are mostly
  interested in homological properties. In fact,
  $\widetilde A\cong \tilde A$ when $A$ is semiperfect and basic.
  Note that the definition of $\widetilde A$ does not depend on any
  choices. In particular, we have an
  identity $\widetilde{A^\op}=\widetilde{A}^\op$.

 Next we discuss some ring theoretic properties which are preserved
 under the passage from $A$ to its cover $\tilde A$.  Recall that a
 ring is \emph{semiprimary} if it is semilocal and its radical is a
 nilpotent ideal.

\begin{rem}\label{re:semilocal}
  If $A$ is semilocal then $\tilde A$ is semilocal.  This follows from
  the isomorphism
  \[\top\tilde A=\top P\oplus \top Q\iso A/\rad A\oplus (A/\rad A)^\natural=\soc Q\]
which is induced by the morphism $\tilde A=P\oplus Q\to
Q\hookrightarrow\tilde A$ given by left multiplication with
  $\smatrix{0&0\\1&(0,1)}$. Moreover, in this case the inclusion
$A\hookrightarrow\tilde A$ yields the identity $(\rad A)^n= (\rad\tilde A)^n$ for
all $n>1$. 
\end{rem}

\begin{rem}\label{re:perfect}
  If the ring $A$ is left or right perfect then the same holds for
  $\tilde A$. This follows from Remark~\ref{re:semilocal}, since $A$ is
  right perfect if and only if $A$ is semilocal and $\rad A$ is right
  T-nilpotent.
\end{rem}

There is an analogue of Theorem~\ref{th:tildeA} for the \emph{big finitistic dimension}
\[\Findim A:=\sup\{\pdim M\mid M\in\Mod A,\, \pdim M<\infty\}.\] We
use the following fact which is a slight variation of
\cite[Theorem~6.3]{Ba1960}.

\begin{prop}\label{pr:findim-zero}
  For a ring $A$ we have $\Findim A=0$ if and only if $A$ is right
  perfect and $\findim A=0$.
\end{prop}
\begin{proof}
  Suppose that $A$ is right perfect and $\findim A=0$. We need to show
  that every monomorphism $\p\colon M\to N$ between projective
  $A$-modules splits. This holds when $M$ and $N$ are finitely
  generated since $\findim A=0$. Because $A$ is right perfect, any
  projective $A$-module decomposes into a direct sum of finitely
  generated modules and can therefore be written as a filtered colimit
  of finitely generated direct summands.  Choose such a presentation
  $M=\colim_i M_i$. Then $\p$ is a filtered colimit of split
  monomorphisms $\p_i\colon M_i\to N$, and therefore
  $\colim_i\Coker\p_i\cong \Coker\p$ is a filtered colimit of
  projectives. Thus $\Coker\p$ is projective and $\p$ splits. For the
  other implication we refer to \cite{Ba1960}.
\end{proof}

\begin{thm}\label{th:Atilde}
  For a ring $A$ we have \[\Findim \tilde A\geq\Findim A.\] Moreover,
  $\Findim \tilde A^\op=0$ if and only if $A$ is left perfect.
\end{thm}

\begin{proof}
  The first assertion is easily checked as in the proof of
  Theorem~\ref{th:tildeA}. If $\Findim \tilde A^\op=0$, then
  $\tilde A$ is left perfect by Proposition~\ref{pr:findim-zero}, and
  this implies that $A$ is left perfect. For the converse suppose that
  $A$ is left perfect. Then $\tilde A$ is left perfect by
  Remark~\ref{re:perfect}.  Thus $\Findim \tilde A^\op=0$ by
  Theorem~\ref{th:tildeA} and Proposition~\ref{pr:findim-zero}.
\end{proof}

\begin{center}
 \textasteriskcentered \qquad \textasteriskcentered \qquad \textasteriskcentered
\end{center}\smallskip

The preceding results demonstrate a failure of symmetry for the notion of `finite
finitistic dimension', as pointed out in the recent work of  Cummings
\cite{Cu2022}. In particular we have the following examples.

Recall that a noetherian ring is \emph{regular} if all its
finitely generated modules have finite projective dimension.

\begin{cor}
  Let $A$ be a commutative noetherian ring that is regular of infinite Krull dimension.
  Then \[\findim \tilde A=\infty\qquad\text{and}\qquad \findim
    \tilde A^\op=0.\]
\end{cor}
\begin{proof}
  The finitistic dimension $\findim A$ is infinite by Theorem~1.6 and
  Corollary~1.7 in \cite{AB1958}.  Thus the assertion follows from
  Theorem~\ref{th:tildeA}.
\end{proof}

Specific examples of  regular rings of infinite Krull dimension have
been constructed by Nagata; cf.\ \cite[Example~7.2.20]{Kr2022}.

We continue with an example due to Kirkman and Kuzmanovich
\cite{KK1990}. Let $k$ be a field and
consider the quotient  $\La=kQ/I$ of the path algebra $kQ$ given by the quiver
\[ \begin{tikzcd} Q:&\circ \arrow[yshift=0.5ex,"a_i"]{r} & \circ
    \arrow[yshift=-0.5ex,"b_i"]{l}\quad (i\in\bbN)\end{tikzcd} \]
(with $k$-basis given by the paths in $Q$ and multiplication induced
by the composition of paths, where for any pair of paths $\a,\b$ we
write $\b\a$ for the composite when the terminal vertex of $\a$ equals
the initial vertex of $\b$) modulo the ideal $I$ that is generated by
the elements
\[  b_t a_s b_r \quad (r,s,t\in\bbN)\qquad b_t a_s-a_t b_t\quad (t>s)\qquad
  a_r b_r\quad (r\in\bbN).\]
Note that $\La$ is a semiprimary ring with $(\rad\La)^4=0$.

\begin{cor}
  The ring $\tilde\La$ is semiprimary satisfying
  \[\Findim \tilde \La=\findim \tilde \La=\infty\qquad\text{and}\qquad \Findim \tilde \La^\op=\findim \tilde
    \La^\op=0.\]
\end{cor}

\begin{proof}
  From Remark~\ref{re:semilocal} it follows that the ring $\tilde \La$
  is semilocal with $(\rad\tilde \La)^4=0$.  Thus $\tilde\La$ is
  semiprimary.  In \cite{KK1990} it is shown that
  $\findim \La=\infty$. Then the assertion follows from
  Theorem~\ref{th:Atilde}, using that $\La$ is left perfect.
\end{proof}

\end{document}